\documentclass{amsart}
\usepackage{amsmath,amsthm}
\usepackage{amsfonts,amssymb}
\usepackage{enumerate}
\usepackage{graphicx}

\newtheorem{lemma}{Lemma}

\newtheorem{corollary}{Corollary}
\newtheorem{theorem}{Theorem}
\newtheorem{remark}{Remark}

\makeatletter
\@addtoreset{equation}{section}
\makeatother

\def\qed{\hfill $\Box$\smallskip}

\newcommand {\ds} {\displaystyle}
\newcommand {\sgn} {{\rm sign}\,}
\begin{document}

\title[]
{On the error bounds of the Gauss-type quadrature formulae
associated with spaces of parabolic and cubic spline functions with
double equidistant knots}

\author{Geno Nikolov}
\address{Faculty of Mathematics and Informatics, Sofia University
"St. Kliment Ohridski", 5 James Bourchier Blvd., 1164 Sofia,
Bulgaria} \email{geno@fmi.uni-sofia.bg}

\author{Petar B. Nikolov}
\address{Faculty of Pharmacy, Medical University of Sofia, 2
Dunav St., 1000 Sofia, Bulgaria}
\email{p.nikolov@pharmfac.mu-sofia.bg}

\begin{abstract}
In two papers from 1995 P. K\"{o}hler and G. Nikolov showed that
Gauss--type quadrature formulae associated with spaces of spline
functions with equidistant knots are asymptotically optimal in
certain Sobolev classes of functions. In particular, Gauss--type
quadratures associated with the spaces of spline functions of degree
$r-1$ with double equispaced knots are asymptotically optimal
definite quadrature formulae of order $r$ when $r$ is even, and it
is conjectured that the asymptotical optimality property persists
also in the case of odd $r$. For $r=3,\,4$, these quadrature
formulae have been constructed by G. Nikolov, who also proved
estimates for their error constants. The aim of this note is to
refine the estimates for the error constant in the case $r=3$, and
to point out to some error estimates in both cases $r=3$ and $r=4$,
which are easier to evaluate and could be sharper than those which
involve the uniform norm of the $r$-th derivative of the integrand.
\medskip

\medskip

\end{abstract}

\maketitle
\section{Introduction and statement of the results}
A standard way to evaluate approximately  the definite integral
$$
I[f]:=\int\limits_{0}^{1}f(x)\,dx
$$
is to use quadrature formulae, which are linear functionals of the
form
\begin{equation}\label{e1.1}
Q[f]=\sum_{i=1}^{n}a_{i}f(\tau_{i}),\quad 0\leq
\tau_{1}<\cdots<\tau_{n}\leq 1.
\end{equation}
We start with introducing some notation and definitions. Throughout
this paper, $\pi_{m}$ stands for the set of algebraic polynomials of
degree not exceeding $m$. A quadrature formula $Q$ is said to have
algebraic degree of precision $m$ (in short, $ADP(Q)=m$) if $m$ is
the largest non-negative integer such that its remainder functional
$$
R[Q;f]:=I[f]-Q[f]
$$
vanishes on $\pi_m$.

The Sobolev classes of functions $W^{r}_p[0,1]$, $\,(r\in
\mathbb{N},\ p\geq 1)$, are defined by
$$
W^{r}_p[0,1]:=\{f\in C^{r-1}[0,1]\,:\, f^{(r-1)} \mbox{ loc. abs.
cont.},\; \int_{0}^{1}\!|f^{(r)}(t)|^p\,dt<\infty\}
$$
(note that $C^{r}[0,1]\subset W^{r}_p[0,1]$ for every $p\geq 1$).
Henceforth, $\|\cdot\|$ designates the supremum norm in $[0,1]$, and
the usual $L_p[0,1]$-norm is shortly denoted by $\|\cdot\|_p$,
$$
\|f\|_p=\begin{cases}\Big(\int\limits_{0}^{1}|f(t)|^{p}\,dt\Big)^{1/p},
& \text{ if } 1\leq p<\infty,\\
{\rm vraisup}_{t\in [0,1]}\, |f(t)|, & \text{ if } p=\infty.
\end{cases}
$$

If $ADP(Q)=m\geq r-1$ and $f\in W^r_1[0,1]$, then by Peano
representation theorem for linear functionals (cf. \cite{GP:1913}),
the remainder $R[Q;f]$ can be written in the form
\begin{equation}\label{e1.2}
R[Q;f]=\int\limits_0^1K_r(Q;t)f^{(r)}(t)\,dt,
\end{equation}
where $K_r(Q;t)$ is referred to as the $r$-th Peano kernel of $Q$
and is given by
\begin{equation}\label{e1.3}
K_r(Q;t)=\frac{1}{(r-1)!}\,R\big[Q;(\cdot-t)_{+}^{r-1}\big],
\end{equation}
where $(x)_+^{r-1}=\max\{x,0\}^{r-1}$ is the truncated power
function. In literature, $K_r(Q;t)$ is also termed as monospline of
degree $r$. For  quadrature formula $Q$ in \eqref{e1.1} the explicit
form of $K_r(Q;t)$, $t\in [0,1]$, is
\begin{eqnarray}
K_r(Q;t)&&=\frac{(1-t)^r}{r!}-\frac{1}{(r-1)!}\,
\sum_{i=1}^{n}a_{i}(\tau_{i}-t)_{+}^{r-1} \label{e1.4}\\
&&=(-1)^r\Big\{\frac{t^r}{r!}-\frac{1}{(r-1)!}\,
\sum_{i=1}^{n}a_{i}(t-\tau_{i})_{+}^{r-1}\Big\}\,.\label{e1.5}
\end{eqnarray}

If $f\in W^{r}_p[0,1]$, then application of H\"{o}lder's inequality
to \eqref{e1.2} implies the unimprovable error estimate
\begin{equation*}
\big|R[Q;f]\big|\leq c_{r,p}(Q)\,\|f^{(r)}\|_p,\quad \text{where }
c_{r,p}(Q)=\|K_r(Q;\cdot)\|_q,\ \ \frac{1}{p}+\frac{1}{q}=1.
\end{equation*}
Usually, $c_{r,p}(Q)$ is called the error constant of $Q$ in the
Sobolev class $W^{r}_p[0,1]$. In what follows, the subscript "$n$"
in $Q_n$ is used to emphasize that $Q_n$ is an $n$-point quadrature
formula, i.e., a quadrature formula which has $n$ nodes. Quadrature
formulae $Q_n$ with the smallest possible error constant
$c_{r,p}(Q_n)$ are called optimal quadrature formulae in
$W^{r}_p[0,1]$. Without going into details, let us mention that the
existence and uniqueness of optimal quadrature formulae in Sobolev
classes of functions have been established by Bojanov
\cite{BB1,BB2,BB3} and Zhensykbaev \cite{AZ1,AZ2}.

In the present paper we study certain definite quadrature formulae.
A quadrature formula $Q_n$ is said to be definite of order $r$
($r\geq 1$), if there exists a constant $c_{r}(Q_n)\ne 0$ such that
\begin{equation*}
R[Q_n;f]=c_{r}(Q_n)\,f^{(r)}(\xi)
\end{equation*}
for every $f\in C^r[0,1]$ with some $\xi\in [0,1]$ depending on the
integrand $f$. More precisely, $Q_n$ is called positive resp.
negative definite quadrature formula of order $r$ if $c_{r}(Q_n)>0$
resp. $c_{r}(Q_n)<0$. Since $c_{r}(Q_n)=c_{r,\infty}(Q_n)$ if $Q_n$
is positive definite and $c_{r}(Q_n)=-c_{r,\infty}(Q_n)$ if $Q_n$ is
negative definite, $c_{r}(Q_n)$ will also be referred to as the
error constant of $Q_n$. The importance of definite quadrature
formulae of order $r$ stems from the fact that they provide
one-sided approximation to $I[f]$ when $f^{(r)}$ has a permanent
sign in $(0,1)$. The midpoint and the trapezium quadrature formulae
are best-known examples of positive resp. negative definite
quadrature formulae of order two.

Definite $n$-point quadrature formulae of order $r$ with the
smallest positive or the largest negative $c_{r}$ are called optimal
definite quadrature formulae. It is known that optimal definite
quadrature formulae exist and are unique, cf. \cite{KJ1976, GL1979,
GS1972} and \cite[Chapter VII.8]{HB1977}. We denote by $c_{n,r}^{+}$
and $c_{n,r}^{-}$ the error constants of the optimal $n$-point
definite quadrature formulae of order $r$:
\begin{eqnarray*}
&& c_{n,r}^{+}:=\inf\{c_{r}(Q_n)\;:\; Q_n\ \text{ is positive
definite of order } r\},\\
&& c_{n,r}^{-}:=\sup\{c_{r}(Q_n)\;:\; Q_n\ \text{ is negative
definite of order } r\}.
\end{eqnarray*}

In \cite{KN1995} estimates have been established for the error
constants of the Gauss-type quadrature formulae associated with the
spaces of spline functions with double equidistant knots. These
estimates in turn provide bounds for the error constants of the
optimal definite quadrature formulae. Below we restate the main
result from \cite{KN1995}, denoting by $B_r$ the Bernoulli
polynomial of order $r$ with leading coefficient $1/r!$.
\medskip

\noindent \textbf{Theorem A.} (\cite{KN1995}, Theorem 1.1) (a)
\emph{For even $r$ with $2\leq r\leq 2n$, there holds}
$$
c_{n,r}^{+}\leq -\frac{B_r(j/2)}{(n+1-r/2)^r},\quad \textit{ if }\
r=4m+2j,\ j=0,1.
$$

(b) \emph{For even $r$ with $2\leq r\leq 2n-2$, there holds}
$$
c_{n,r}^{-}\geq -\frac{B_r(j/2)}{(n-r/2)^r},\quad \textit{ if }\
r=4m+2-2j,\ j=0,1.
$$

(c) \emph{For odd $r$ with $1\leq r\leq 2n-1$, there holds}
$$
c_{n,r}^{+}\leq \frac{\|B_r\|}{(n-(r-1)/2)^r}\quad \textit{ and }\ \
c_{n,r}^{-}\geq -\frac{\|B_r\|}{(n-(r-1)/2)^r}.
$$
\medskip

Comparison with results of Lange \cite{GL1977} shows that Gauss-type
quadrature formulae associated with the spaces of spline functions
of degree $r-1$ with double equidistant knots are asymptotically
optimal definite quadrature formulae of order $r$ when $r$ is even,
and it is conjectured that the asymptotical optimality property
persists also in the case of odd $r$. Two particular cases of
Theorem~A relevant to the object of this paper are
\begin{eqnarray}
&&c_{n+1,3}^{+}\leq \frac{\sqrt{3}}{216n^3},\quad c_{n+1,3}^{-}\geq
-\frac{\sqrt{3}}{216n^3},\label{e1.6}\\
&&c_{n+1,4}^{+}\leq\frac{1}{720n^4}\,.\label{e1.7}
\end{eqnarray}
The right-hand sides of the inequalities in \eqref{e1.6} and
\eqref{e1.7} are in fact bounds for the error constants of
Gauss-type quadrature formulae associated with the linear spaces of
spline functions $S_{n,3}$ and $S_{n,4}$, respectively, where for
$r\geq 3$ and $n\geq 2$,
\begin{equation}\label{e1.8}
\begin{split}
&S_{n,r}=\{f\;:\; f\in C^{r-3}[0,1],\;f_{|(x_{k},x_{k+1})}\in
\pi_{r-1},\ k=0,\ldots,n-1\},\\
&x_{k}=x_{k,n}:=\frac{k}{n},\quad k=0,\ldots,n.
\end{split}
\end{equation}
The functions $\{1,x,x^2,(x-x_1)_{+},(x-x_1)_{+}^2,\ldots,
(x-x_{n-1})_{+},(x-x_{n-1})_{+}^2\}$ form a basis for $S_{n,3}$,
therefore $\dim S_{n,3}=2n+1$ and the Gauss-type quadratures
associated with $S_{n,3}$ are left and right $(n+1)$-point Radau
quadrature formulae. These quadrature formulae were found, among
others, in \cite{GN1993}.\medskip

\noindent \textbf{Theorem B. }(\cite[Theorem~2]{GN1993}) \emph{The
right Radau quadrature formula associated with the space of
parabolic splines $S_{n,3}$ is}
$$
Q_{n+1}^{R,r}[f]=\sum_{k=0}^{n-1}a_{k,n}f\Big(\frac{k+\theta_k}{n}\Big)
+a_{n,n}f(1)\approx I[f],
$$
\emph{where $\theta_0=1/3$,}
\begin{eqnarray*}
&&\theta_{k}=\frac{1-\theta_{k-1}}{5-6\theta_{k-1}},\quad
k=1,\ldots,n-1, \\
&& a_{k,n}=\frac{1}{6n\theta_k(1-\theta_k)},\quad k=0,\ldots,n-1,
\end{eqnarray*}
and
$$
a_{n,n}=\frac{2-3\theta_{n-1}}{6n(1-\theta_{n-1})}.
$$
$Q_{n+1}^{R,r}$ \emph{ is negative definite quadrature formula of
order three with error constant}
$$
c_3\big(Q_{n+1}^{R,r}\big)=-\frac{\sqrt{3}}{216n^3}+O(n^{-4}).
$$
\begin{remark}\label{r1}
The left Radau quadrature formula $Q_{n+1}^{R,l}$ associated with
$S_{n,3}$ is obtained from $Q_{n+1}^{R,r}$ by reflection, i.e.
$Q_{n+1}^{R,l}[f(\cdot)]=Q_{n+1}^{R,r}[f(1-\cdot)]$. Clearly,
$Q_{n+1}^{R,l}$ is positive definite of order three and
$c_3\big(Q_{n+1}^{R,l}\big)=-c_3\big(Q_{n+1}^{R,r}\big)$.
\end{remark}

Since $\dim S_{n,4}=2n+2$, associated with $S_{n,4}$ are the
$(n+1)$-point Gauss quadrature formula $Q_{n+1}^{G}$ and the
$(n+2)$-point Lobatto quadrature formula $Q_{n+2}^{Lo}$. These
quadrature formulae were investigated in \cite{GN1996}. The
following theorem gives the construction and summarizes some of the
properties of $Q_{n+1}^{G}$ (cf. \cite[Section 2]{GN1996}):\medskip

\noindent \textbf{Theorem C.} \emph{Let } $\displaystyle
Q_{n+1}^{G}[f]=\sum_{i=1}^{n+1}a_{i,n+1}^{G}f(\tau_{i,n+1}^{G}),\
0<\tau_{1,n+1}^{G}<\cdots<\tau_{n+1,n+1}^{G}<1$, \emph{be the Gauss
quadrature formula associated with $S_{n,4}$, i.e. determined
uniquely by the property $I[f]=Q_{n+1}^{G}[f]$ for every $f\in
S_{n,4}$. Then:}

(a) $~Q_{n+1}^{G}$ \emph{is symmetrical: }
$a_{k,n+1}^{G}=a_{n+2-k,n+1}^{G}$ \emph{ and }
$\tau_{k,n+1}^{G}=1-\tau_{n+2-k,n+1}^{G}\,$ \emph{ for }
$k=1,\ldots,n+1$.

(b) \emph{~Let $a_{i,n+1}^{G}=\ds\frac{\delta_i}{n}$,
$\tau_{i,n+1}^{G}=\ds\frac{i-\theta_i}{n}$ for $i=1,\ldots,[n/2]+1$.
Then the sequences $\{\delta_i\}$ and $\{\theta_i\}$ are determined
by $\ds \delta_1=\frac{16}{27}$, $\ds \theta_1=\frac{3}{4}$ and, for
$i=1,\ldots,[n/2]-1$, by the recurrence relations}
$$
\theta_{i+1}=\frac{1-\delta_i(1-\theta_i)^2(5\theta_i+1)}
{1-\delta_i(1-\theta_i)^2(4\theta_i+1)},\quad
\delta_{i+1}=\frac{1-\delta_i(1-\theta_i)^2(4\theta_i+1)}{\theta_{i+1}^2}\,.
$$
\emph{If $n$ is even $(n=2m)$, then $\theta_{m+1}=1$ and
$\delta_{m+1}=1-2\delta_m(1-\theta_m)^2(2\theta_m+1)$; if $n$ is odd
$(n=2m-1)$, then
$\delta_m=1-\delta_{m-1}(1-\theta_{m-1})^2(2\theta_{m-1}+1)$ and
$\theta_m$ is the greater root of the equation}
$$
\theta_m(1-\theta_m)=\frac{\delta_{m-1}\theta_{m-1}(1-\theta_{m-1})^2}
{1-\delta_{m-1}(1-\theta_{m-1})^2(2\theta_{m-1}+1)}.
$$

(c) \emph{~$Q_{n+1}^{G}$ is positive definite quadrature formula of
order four and its error constant $c_{4}(Q_{n+1}^{G})$ obeys the
representation}
$$
c_{4}(Q_{n+1}^{G})=\frac{1}{720n^4}-\frac{1}{12}\sum_{i=1}^{[(n+1)/2]}
a_{i,n+1}^{G}(x_{i-1}-\tau_{i,n+1}^{G})^2(x_{i}-\tau_{i,n+1}^{G})^2\,.
$$
\emph{For all $n\geq 4$ there holds}
\begin{equation}\label{e1.9}
\frac{1}{720n^4}-\frac{1}{551.9775n^5}\leq c_{4}(Q_{n+1}^{G})\leq
\frac{1}{720n^4}-\frac{1}{552n^5}\,.
\end{equation}

(d) \emph{~Let $f\in C^4[0,1]$. Then $R[Q_{n+1}^{G};f]=o(n^{-4})$ if
and only if $f^{\prime\prime\prime}(0)=f^{\prime\prime\prime}(1)$.
Moreover, if
$\sgn\{f^{\prime\prime\prime}(1)-f^{\prime\prime\prime}(0)\}
=\epsilon\ne 0$, then there exists $n_0\in\mathbb{N}$ such that
$\epsilon R[Q_{n+1}^{G};f]\geq 0$ for all $n\geq n_0$}.\smallskip

(e) \emph{If $f\in W_1^4[0,1]$ and $f^{(4)}\ge 0$ a.e. in $[0,1]$,
then for all $n\geq 2$,}
$$
0\leq R\big[Q_{2n+1}^G;f\big]\leq R\big[Q_{n+1}^G;f\big].
$$
\begin{remark}\label{r2}
In \cite{GN1996} recurrence formulae have been proposed also for
computation of the weights and nodes of the Lobatto quadrature
formula $Q_{n+2}^{Lo}$ associated with $S_{n,4}$, which is negative
definite of order four. However, unlike the case with $Q_{n+1}^{G}$,
this procedure is of numerical nature, as it requires determination
of an initial parameter, cf. \cite[Theorem~2.5]{GN1996}.
\end{remark}

Our first goal in this paper is to prove properties of the Radau
quadrature formulae associated with $S_{n,3}$, which are the
analogues of those of $Q_{n+1}^{G}$, presented in parts (c), (d) and
(e) of Theorem~C.
\begin{theorem}\label{t1}
Let $Q_{n+1}^{R,l}$ and $Q_{n+1}^{R,r}$ be the $(n+1)$-point left
and right Radau quadrature formulae associated with $S_{n,3}$, i.e.,
determined uniquely by the property $R\big[Q_{n+1}^{R,l};f\big]=
R\big[Q_{n+1}^{R,r};f\big]=0$ for every $f\in S_{n,3}$.
Then:
\medskip

{\rm (a)~~} The error constants of $Q_{n+1}^{R,l}$ and
$Q_{n+1}^{R,r}$ are given by
\begin{equation}\label{e1.10}
c_{3}\big(Q_{n+1}^{R,l}\big)=-c_{3}\big(Q_{n+1}^{R,r}\big)
=\frac{\sqrt{3}}{216n^3}-\frac{\sqrt{3}}{108n^4}\,\sum_{k=0}^{n-1}
\frac{1}{(2+\sqrt{3})^{2k+1}+1}\,.
\end{equation}
With * standing for both $r$ and $l$,  the following inequalities
hold true for all $n\geq 4$:
\begin{equation}\label{e1.11}
\frac{\sqrt{3}}{216n^3}-\frac{1}{269.13n^4}
<\big|c_{3}\big(Q_{n+1}^{R,*}\big)\big|
<\frac{\sqrt{3}}{216n^3}-\frac{1}{269.14n^4}.
\end{equation}

{\rm (b)~~} Let $f\in C^3[0,1]$. With * standing for both $r$ and
$l$, $R[Q_{n+1}^{R,*};f]=o(n^{-3})$ as $n\to\infty$ if and only if
$f^{\prime\prime}(0)=f^{\prime\prime}(1)$. Moreover, if
$\sgn\{f^{\prime\prime}(1)-f^{\prime\prime}(0)\} =\epsilon\ne 0$,
then there exists $n_0\in\mathbb{N}$ such that $\epsilon\,
R[Q_{n+1}^{R,r};f]\leq 0$ and $\epsilon\,R[Q_{n+1}^{R,l};f]\geq 0$
for all $n\geq n_0$.\smallskip

{\rm (c)~~} If $f\in W_1^3[0,1]$ and $f^{\prime\prime\prime}\ge 0$
a.e. in $[0,1]$, then for all $n\geq 2$,
$$
0\geq R\big[Q_{2n+1}^{R,r};f\big]\geq
R\big[Q_{n+1}^{R,r};f\big]\quad \textit{ and }\ 0\leq
R\big[Q_{2n+1}^{R,l};f\big]\leq R\big[Q_{n+1}^{R,l};f\big]\,.
$$
\end{theorem}
\begin{remark}\label{r3}
Theorem~C and Theorem~\ref{t1} provide the following improvement of
the estimates \eqref{e1.6} and \eqref{e1.7} for the the error
constants of the optimal positive definite quadrature formulae of
orders two and three with $n+1$ nodes, $n\geq 4$:
\begin{eqnarray*}
&&c_{n+1,3}^{+}\leq c_{3}\big(Q_{n+1}^{R,l}\big)<
\frac{\sqrt{3}}{216n^3}-\frac{1}{269.14n^4}, \\
&&c_{n+1,4}^{+}\leq c_{4}\big(Q_{n+1}^{G}\big)<
\frac{1}{720n^4}-\frac{1}{552n^5}.
\end{eqnarray*}
\end{remark}
As a consequence we have the following
\begin{corollary}\label{c1}
{\rm (a)~~} If $f\in C^3[0,1]$ and $f^{\prime\prime\prime}\geq 0$ in
$[0,1]$, then for all $n\geq 4$,
\begin{equation}\label{e1.12}
\begin{split}
&0\leq R\big[Q_{n+1}^{R,l};f\big]\leq
\Big(\frac{\sqrt{3}}{216n^3}-\frac{1}{269.14n^4}\Big)
\|f^{\prime\prime\prime}\|,\\
&0\geq R\big[Q_{n+1}^{R,r};f\big]\geq
-\Big(\frac{\sqrt{3}}{216n^3}-\frac{1}{269.14n^4}\Big)
\|f^{\prime\prime\prime}\|\,.
\end{split}
\end{equation}

{\rm (b)~~} If $f\in C^4[0,1]$ and $f^{(4)}\geq 0$ in $[0,1]$, then
for all $n\geq 4$
\begin{equation}\label{e1.13}
0\leq R\big[Q_{n+1}^{G};f\big]\leq \Big(
\frac{1}{720n^4}-\frac{1}{552n^5}\Big)\|f^{(4)}\|\,.
\end{equation}
\end{corollary}
Alternative error estimates are provided by the following theorem:
\begin{theorem}\label{t2}
{\rm (a)~~} If $f\in C^3[0,1]$ and $f^{\prime\prime\prime}\geq 0$ in
$[0,1]$, then for all $n\geq 2$,
\begin{equation}\label{e1.14}
\begin{split}
&0\leq R\big[Q_{n+1}^{R,l};f\big]\leq \frac{\sqrt{3}}{108n^3}\,
\Big(f^{\prime\prime}(1)-f^{\prime\prime}(0)\Big),\\
&0\geq R\big[Q_{n+1}^{R,r};f\big]\geq \frac{\sqrt{3}}{108n^3}\,
\Big(f^{\prime\prime}(0)-f^{\prime\prime}(1)\Big)\,.
\end{split}
\end{equation}

{\rm (b)~~} If $f\in C^4[0,1]$ and $f^{(4)}\geq 0$ in $[0,1]$, then
for all $n\geq 2$,
\begin{equation}\label{e1.15}
0\leq R\big[Q_{n+1}^{G};f\big]\leq \frac{1}{384n^4}\,
\Big(f^{\prime\prime\prime}(1)-f^{\prime\prime\prime}(0)\Big)\,.
\end{equation}
\end{theorem}
Since the supremum norms of $f^{\prime\prime\prime}$ and $f^{(4)}$
may be  not accessible or difficult to evaluate, evidently the error
bounds in Theorem~\ref{t2} are easier to apply than those in
Corollary~\ref{c1}. Even in the cases when
$\|f^{\prime\prime\prime}\|$ or $\|f^{(4)}\|$ is known, it can still
happen that the estimates in Theorem~\ref{t2} are superior to those
from Corollary~\ref{c1}.

Before concluding this section, we find appropriate to briefly
mention a few more facts about Peano kernel representation of the
remainders of quadrature formulae, for more details the reader is
referred to \cite{HB1977}.

It follows from \eqref{e1.3} that the requirement $K_r(Q;u)=0$ for
some $u\in (0,1)$ is equivalent to $I[f_{u}]=Q[f_{u}]$, where
$f_{u}(x)=(x-u)_{+}^{r-1}$. Hence, in order that $Q$ evaluates to
the exact value definite integrals of functions from a linear space
of splines of degree $r-1$ with maximal dimension, it is necessary
that the monospline $K_r(Q;\cdot)$ has the maximal possible number
of zeros in $(0,1)$. The problem of the existence and uniqueness of
monosplines satisfying boundary conditions and having maximal number
of prescribed zeros in $(0,1)$ (the fundamental theorem of algebra
for monosplines) has been resolved by Karlin and Micchelli
\cite{KM1972}. Quadrature formulae corresponding to monosplines of
the form \eqref{e1.4}-\eqref{e1.5} with maximal number of
preassigned zeros in $(0,1)$ are called Gauss-type quadratures
associated with the space of spline functions of degree $r-1$ having
knots at these zeros. The results in \cite{KM1972} assert that
Gauss-type quadratures for spaces of splines exist and are unique,
and as in the case of classical Gauss-type quadratures associated
with spaces of algebraic polynomials, all their weights are
positive.

We finally point out that, in view of \eqref{e1.2}, a quadrature
formula $Q$ is definite of order $r$ if and only if $ADP(Q)=r-1$ and
$K_r(Q;t)$ does not change its sign in $(0,1)$. Therefore, all zeros
of $K_r(Q;\cdot)$ in $(0,1)$ must have even multiplicities.

Theorem~\ref{t1} is proved in the next section, and in section~3 we
present the proof of Theorem~\ref{t2}.
\section{Proof of Theorem~\ref{t1}}
In view of Remark~\ref{r1}, it suffices to prove only the claims of
Theorem~\ref{t1} concerning $Q_{n+1}^{R,r}$. We denote the right
Radau quadrature formula associated with $S_{n,3}$ by
$$
Q_{n+1}^{R,r}[f]=\sum_{k=0}^{n-1}a_{k}\,f(\tau_k)+a_n\,f(1),\quad
0<\tau_0<\cdots<\tau_{n-1}<1,
$$
where, for the sake of simplicity, we skip the second indices in the
weights and nodes (we also write $x_k=k/n$, $k=0,\ldots,n$, see
\eqref{e1.8}).

According to Theorem~B, we have
\begin{equation}\label{e3.2}
\tau_k=\frac{k+\theta_k}{n},\quad k=0,\ldots,n-1,
\end{equation}
with
\begin{equation}\label{e3.3}
\theta_0=\frac{1}{3},\quad
\theta_{k}=\frac{1-\theta_{k-1}}{5-6\theta_{k-1}},\quad
k=1,\ldots,n-1,
\end{equation}
\begin{equation}\label{e3.4}
a_{k}=\frac{1}{6n\theta_{k}(1-\theta_k)},\quad 0\leq k\leq n-1,
\end{equation}
and
\begin{equation}\label{e3.5}
a_{n}=\frac{2-3\theta_{n-1}}{6n(1-\theta_{n-1})}.
\end{equation}

\begin{lemma}\label{l1}
The sequence $\{\theta_k\}$ in \eqref{e3.3} has the explicit
representation
$$
\theta_{k}=\frac{s_k}{s_k+s_{k+1}},\quad s_k=(2+\sqrt{3})^k
+(2-\sqrt{3})^k\,,\quad k\in \mathbb{N}\,.
$$
\end{lemma}
\begin{proof}
We apply induction with respect to $k$. The statement is true for
$k=0$, since $s_0=2$ and $s_1=4$. Assuming $\ds
\theta_{k-1}=\frac{s_{k-1}}{s_{k-1}+s_k}$ for some $k\in\mathbb{N}$,
then
$$
\theta_k=\frac{1-\theta_{k-1}}{5-6\theta_{k-1}}
=\frac{s_k}{5s_k-s_{k-1}}\stackrel{\mbox{?}}{=}\frac{s_k}{s_k+s_{k+1}}.
$$
The last equality follows from the identity $s_{k-1}+s_{k+1}=4s_k$,
which is verified using $(2\pm\sqrt{3})^2+1=4(2\pm\sqrt{3})$. This
accomplishes the induction step and thereby the proof of
Lemma~\ref{l1}. The proposed method of proof does not give a clue
about the way the explicit form of the solution of this recurrence
equation was deduced. Equations like \eqref{e3.3} are called Riccati
difference equations, see e.g. \cite{LB1955} for a general approach
to their solutions.
\end{proof}

For $f\in C^3[0,1]$ the remainder of $Q_{n+1}^{R,r}$ admits the
representation
$$
R\big[Q_{n+1}^{R,r};f\big]=\int\limits_0^1
K_3(Q_{n+1}^{R,r};t)f^{\prime\prime\prime}(t)\,dt,
$$
where, according to \eqref{e1.5},
$$
K_3(Q_{n+1}^{R,r};t)=-\frac{t^3}{6}+\frac{1}{2}
\sum_{i=0}^{n-1}a_i(t-\tau_i)_{+}^2\leq 0,\quad t\in (0,1).
$$
The zeros of $K_3(Q_{n+1}^{R,r};\cdot)$ in $(0,1)$ are
$\{x_k\}_{k=1}^{n-1}$, and each of them is double. The error
constant of $Q_{n+1}^{R,r}$ is given by
\begin{equation}\label{e3.6}
c_3\big(Q_{n+1}^{R,r}\big)=\int\limits_0^1 K_3(Q_{n+1}^{R,r};t)\,dt
=\sum_{k=0}^{n-1}\int\limits_{x_k}^{x_{k+1}}K_3(Q_{n+1}^{R,r};t)\,dt
=:\sum_{k=0}^{n-1}I_k\,.
\end{equation}
Clearly,
\begin{equation}\label{e3.7}
I_k=\frac{1}{24}\big(x_k^4-x_{k+1}^4\big) +\frac{1}{6}\,J_k,
\end{equation}
where
$$
J_k=\sum_{i=0}^{k}a_i(x_{k+1}-\tau_i)^3
-\sum_{i=0}^{k-1}a_i(x_{k}-\tau_i)^3\,.
$$
By using $x_{k+1}-x_k=1/n$, we obtain
\begin{equation*}
\begin{split}
J_k=& \frac{1}{n}\sum_{i=0}^{k}
a_i\Big[(x_{k+1}-\tau_i)^2+(x_{k+1}-\tau_i)(x_{k}-\tau_i)
+(x_{k}-\tau_i)^2\Big]-a_k(\tau_k-x_k)^3\\
=&\frac{1}{n}\sum_{i=0}^{k}a_i
\Big[2(x_{k+1}-\tau_i)^2-\frac{1}{n}(x_{k+1}-\tau_i)
+(x_{k}-\tau_i)^2\Big]-a_k(\tau_k-x_k)^3\\
=&\frac{2}{n}Q_{n+1}^{R,r}\big[(x_{k+1}-\cdot)_{+}^2\big]
-\frac{1}{n^2}Q_{n+1}^{R,r}\big[(x_{k+1}-\cdot)_{+}\big]
+\frac{1}{n}Q_{n+1}^{R,r}\big[(x_{k}-\cdot)_{+}^2\big]\\
&+a_k(\tau_k-x_k)^2(x_{k+1}-\tau_k)\,.
\end{split}
\end{equation*}
Since $Q_{n+1}^{R,r}[f]=I[f]$ for every $f\in S_{n,3}$, we have
\begin{equation*}
\begin{split}
J_k=&\frac{2}{n}I\big[(x_{k+1}-\cdot)_{+}^2\big]
-\frac{1}{n^2}I\big[(x_{k+1}-\cdot)_{+}\big]
+\frac{1}{n}I\big[(x_{k}-\cdot)_{+}^2\big]\\
&+a_k(\tau_k-x_k)^2(x_{k+1}-\tau_k)\\
=&\frac{2}{3n}x_{k+1}^3-\frac{1}{2n^2}x_{k+1}^2+\frac{1}{3n}x_k^3
+a_k(\tau_k-x_k)^2(x_{k+1}-\tau_k)\,.
\end{split}
\end{equation*}
Substituting this expression for $J_k$ in \eqref{e3.7} and replacing
$x_k$, $x_{k+1}$, $a_k$ and $\tau_k$ using \eqref{e1.8},
\eqref{e3.2} and \eqref{e3.4}, we obtain
\begin{equation}\label{e3.8}
I_k=\frac{2\theta_k-1}{72n^4}\,.
\end{equation}
By using Lemma~\ref{l1}, we find
\begin{equation*}
\begin{split}
1-2\theta_k&=\frac{s_{k+1}-s_k}{s_{k+1}+s_k}=
\frac{(\sqrt{3}+1)(2+\sqrt{3})^{k}-(\sqrt{3}-1)(2-\sqrt{3})^k}
{(3+\sqrt{3})(2+\sqrt{3})^k+(3-\sqrt{3})(2-\sqrt{3})^k}\\
&=\frac{\sqrt{3}}{3}\,\frac{(2+\sqrt{3})^{k}-\frac{\sqrt{3}-1}{\sqrt{3}+1}
(2-\sqrt{3})^k}{(2+\sqrt{3})^{k}+\frac{\sqrt{3}-1}{\sqrt{3}+1}
(2-\sqrt{3})^k}=\frac{\sqrt{3}}{3}\,
\frac{(2+\sqrt{3})^{k}-(2-\sqrt{3})^{k+1}}
{(2+\sqrt{3})^{k}+(2-\sqrt{3})^{k+1}}\\
&=\frac{\sqrt{3}}{3}\,\Big(1-\frac{2(2-\sqrt{3})^{k+1}}
{(2+\sqrt{3})^{k}+(2-\sqrt{3})^{k+1}}\Big)\\
&=\frac{\sqrt{3}}{3}\,\Big(1-\frac{2}{(2+\sqrt{3})^{2k+1}+1}\Big)\,.
\end{split}
\end{equation*}
By plugging this expression in \eqref{e3.8}, we arrive at
\begin{equation}\label{e3.9}
I_k=-\frac{\sqrt{3}}{216n^4}\,\Big(1-\frac{2}{(2+\sqrt{3})^{2k+1}+1}\Big),\quad
k=0,\ldots,n-1\,.
\end{equation}
The representation \eqref{e1.10} of $c_3(Q_{n+1}^{R,r})$ in
Theorem~\ref{t1}(a) now follows from \eqref{e3.6} and \eqref{e3.9}.
As was already mentioned, $c_3(Q_{n+1}^{R,l})=-c_3(Q_{n+1}^{R,r})$.
The two-sided estimates \eqref{e1.11} are derived using the
inequalities
$$
\sum_{k=0}^{3}\!\frac{1}{(2\!+\!\sqrt{3})^{2k+1}\!+\!1}\leq\!
\sum_{k=0}^{n-1\!}\frac{1}{(2\!+\!\sqrt{3})^{2k+1}\!+\!1}\leq\!
\sum_{k=0}^{3}\!\frac{1}{(2\!+\!\sqrt{3})^{2k+1}\!+\!1}+\!
\sum_{k=4}^{\infty}\!\frac{1}{(2\!+\!\sqrt{3})^{2k+1}}.
$$
With this Theorem~\ref{t1}(a) is proved, and we proceed with the
proof of part (b). If $f\in C^3[0,1]$, then by the mean value
theorem,
$$
R\big[Q_{n+1}^{R,r};f\big]= \sum_{k=0}^{n-1}
\int\limits_{x_k}^{x_{k+1}}
K_3(Q_{n+1}^{R,r};t)f^{\prime\prime\prime}(t)\,dt
=\sum_{k=0}^{n-1}I_k\,f^{\prime\prime\prime}(\xi_k)
$$
with $\xi_k\in (x_k,x_{k+1})$, $k=0,\ldots,n-1$. We split the last
sum into two parts:
$$
R\big[Q_{n+1}^{R,r};f\big]=-\frac{\sqrt{3}}{216n^3}
\sum_{k=0}^{n-1}\frac{1}{n}\,f^{\prime\prime\prime}(\xi_k)+
\sum_{k=0}^{n-1}\Big[I_k+\frac{\sqrt{3}}{216n^4}\Big]
\,f^{\prime\prime\prime}(\xi_k)=:A+B.
$$
The sum in $A$ is a Riemann sum for the continuous (hence
integrable) function $f^{\prime\prime\prime}$ on $[0,1]$, therefore
$$
A=-\frac{\sqrt{3}}{216n^3}\,
\big(f^{\prime\prime}(1)-f^{\prime\prime}(0)\big)+o(n^{-3})\,.
$$
For $B$ we have, in view of \eqref{e3.9},
$$
B=\frac{\sqrt{3}}{108n^4}\,\sum_{k=0}^{n-1}
\frac{1}{(2+\sqrt{3})^{2k+1}+1}f^{\prime\prime\prime}(\xi_k)=O(n^{-4}).
$$
Hence,
$$
R\big[Q_{n+1}^{R,r};f\big]=A+B=-\frac{\sqrt{3}}{216n^3}\,
\big(f^{\prime\prime}(1)-f^{\prime\prime}(0)\big)+o(n^{-3}),
$$
which proves Theorem~\ref{t1}(b) for the remainders of
$Q_{n+1}^{R,r}$.

For the proof of Theorem~\ref{t1}(c) we need the estimate for the
number of zeros of a spline function in a given interval $(a,b)$,
provided by the Budan-Fourier theorem for splines. For a real-valued
function $f$ defined on the finite interval $[a,b]$, $Z_f(a,b)$
stands for the total number of the zeros of $f$ in $(a,b)$ counted
with their multiplicities. By $S^{-}(a_1,a_2,\ldots,a_m)$ and
$S^{+}(a_1,a_2,\ldots,a_m)$ we denote the number of strong and weak
sign changes, respectively, in the finite sequence of real numbers
$a_1,a_2,\ldots,a_m$.
\begin{lemma} [\cite{KN1995a}, Theorem 2.1]\label{l2}
If $f$ is a polynomial spline function of exact degree $r$ on
$(a,b)$ (i.e. of degree $r$ with $f^{(r)}(t)\ne 0$ for some $t\in
(a,b)$ with finitely many (active) knots in $(a,b)$, all simple),
then
\begin{equation*}
\begin{split}
Z_f(a,b)\leq &Z_{f^{(r)}}(a,b)+
S^{-}(f(a),f^{\prime}(a),\ldots,f^{(r-1)}(a),f^{(r)}(\sigma+))\\
&-S^{+}(f(b),f^{\prime}(b),\ldots,f^{(r-1)}(b),f^{(r)}(\tau-)),
\end{split}
\end{equation*}
where $[\sigma,\tau]\subset [a,b]$ is the largest interval such that
$f^{(r)}(\sigma+)\ne 0$ and $f^{(r)}(\tau-)\ne 0$.
\end{lemma}

The difference $s(t)=K_3(Q_{n+1}^{R,r};t)-K_3(Q_{2n+1}^{R,r};t)$ of
the third Peano kernels of the right Radau quadrature formulae
associated with $S_{n,3}$ and $S_{2n,3}$ is a spline function of
degree two with $3n$ knots in $(0,1)$, which has double zeros at the
points $x_{k}=k/n$, $k=1,\ldots,n-1$. In view of \eqref{e1.4} and
\eqref{e1.5}, $s$ can be represented in two alternative ways,
\begin{equation}\label{e3.10}
s(t)=\frac{1}{2}\Big(\sum_{k=0}^{n-1}a_{k,n}(t-\tau_{k,n})_{+}^2
-\sum_{k=0}^{2n-1}a_{k,2n}(t-\tau_{k,2n})_{+}^2 \Big)\,,
\end{equation}
\begin{equation}\label{e3.11}
s(t)=\frac{1}{2}\Big(\sum_{k=0}^{2n}a_{k,2n}(\tau_{k,2n}-t)_{+}^2-
\sum_{k=0}^{n}a_{k,n}(\tau_{k,n}-t)_{+}^2\Big)\,.
\end{equation}
Recall that all weights $a_{k,n}$ and $a_{k,2n}$ of Radau quadrature
formulae are positive, therefore $s(t)$ is a spline function of
exact degree two. Indeed,
$$
s^{\prime\prime}(t)=\sum_{k=0}^{n-1}a_{k,n}(t-\tau_{k,n})_{+}^0
-\sum_{k=0}^{2n-1}a_{k,2n}(t-\tau_{k,2n})_{+}^0
$$
is a piecewise constant function whose $n$ positive jumps cannot be
canceled out by the $2n$ negative jumps. This observation implies
also that the number of sign changes of $s^{\prime\prime}$ in
$(0,1)$ does not exceed $2n$, i.e.,
\begin{equation}\label{e3.12}
Z_{s^{\prime\prime}}(0,1)\leq 2n\,.
\end{equation}

By Theorem~B, $\ds
\tau_{0,2n}=\frac{1}{6n}<\tau_{0,n}=\frac{1}{3n}$, therefore
$s^{\prime\prime}(\tau_{0,2n}+)=-a_{0,2n}<0$ while $s(t)\equiv 0$
for $t\in [0,\tau_{0,2n})$. From $\tau_{n,n}=\tau_{2n,2n}=1$ and
\eqref{e3.11} we obtain $s^{\prime\prime}(1-)=a_{2n,2n}-a_{n,n}$. We
shall show that $a_{2n,2n}-a_{n,n}\ne 0$, in fact,
\begin{equation}\label{e3.13}
a_{2n,2n}-a_{n,n}<0\,.
\end{equation}
From \eqref{e3.5} and Lemma~\ref{l1} we find
$$
a_{n,n}=\frac{2s_n-s_{n-1}}{6n\,s_n},
$$
hence \eqref{e3.13} is equivalent to inequality
$$
2\,\frac{s_{n-1}}{s_n}-\frac{s_{2n-1}}{s_{2n}}<2,
$$
which obviously is true since $0<s_{k-1}<s_k$, $k\in \mathbb{N}$.

Lemma~\ref{l2} applied with $r=2$, $f=s$, $[a,b]=[0,1]$ and
$[\sigma,\tau]=[\tau_{0,2n},1]$ yields
\begin{equation*}
\begin{split}
Z_s(0,1)&\leq Z_{s^{\prime\prime}}(0,1)+
S^{-}\big(s(0),s^{\prime}(0),s^{\prime\prime}(\tau_{0,2n}+)\big)
-S^{+}\big(s(1),s^{\prime}(1),s^{\prime\prime}(1-)\big)\\
&\leq 2n+S^{-}\big(0,0,-a_{0,2n}\big)-
S^{+}\big(0,0,a_{2n,2n}-a_{n,n}\big)\\
&\leq 2n-2\,.
\end{split}
\end{equation*}
Recalling that $s$ has double zeros at the points $k/n$,
$k=1,\ldots,n-1$, we conclude that $s$ has no other zeros in
$(0,1)$. Since $s(1)=s^{\prime}(1)=0$ and $s^{\prime\prime}(1-)<0$,
it follows that $s(t)\leq 0$, $t\in (0,1)$, i.e.,
$$
K_3(Q_{n+1}^{R,r};t)\leq K_3(Q_{2n+1}^{R,r};t)\leq 0,\quad t\in
(0,1).
$$
If $f\in W^3_1[0,1]$ and $f^{\prime\prime\prime}(t)\geq 0$ a.e. in
$[0,1]$, then
$$
R\big[Q_{n+1}^{R,r};f\big]-R\big[Q_{2n+1}^{R,r};f\big]=\int\limits_0^1
s(t)f^{\prime\prime\prime}(t)\,dt\leq 0,
$$
and consequently
$$
R\big[Q_{n+1}^{R,r};f\big] \leq R\big[Q_{2n+1}^{R,r};f\big]\leq 0.
$$
With this Theorem~\ref{t1}(c) is proved. \qed

Figure~1 illustrates the situation when $n=4$. Its left part depicts
the graphs of the third Peano kernels of the $5$-point and $9$-point
right Radau quadrature formulae. We observe that the difference of
the two Peano kernels, depicted on the right, vanishes on the
interval $[0,\tau_{0,9}]$.

\begin{figure}[htp]
\centering
\includegraphics[scale=0.3,clip]{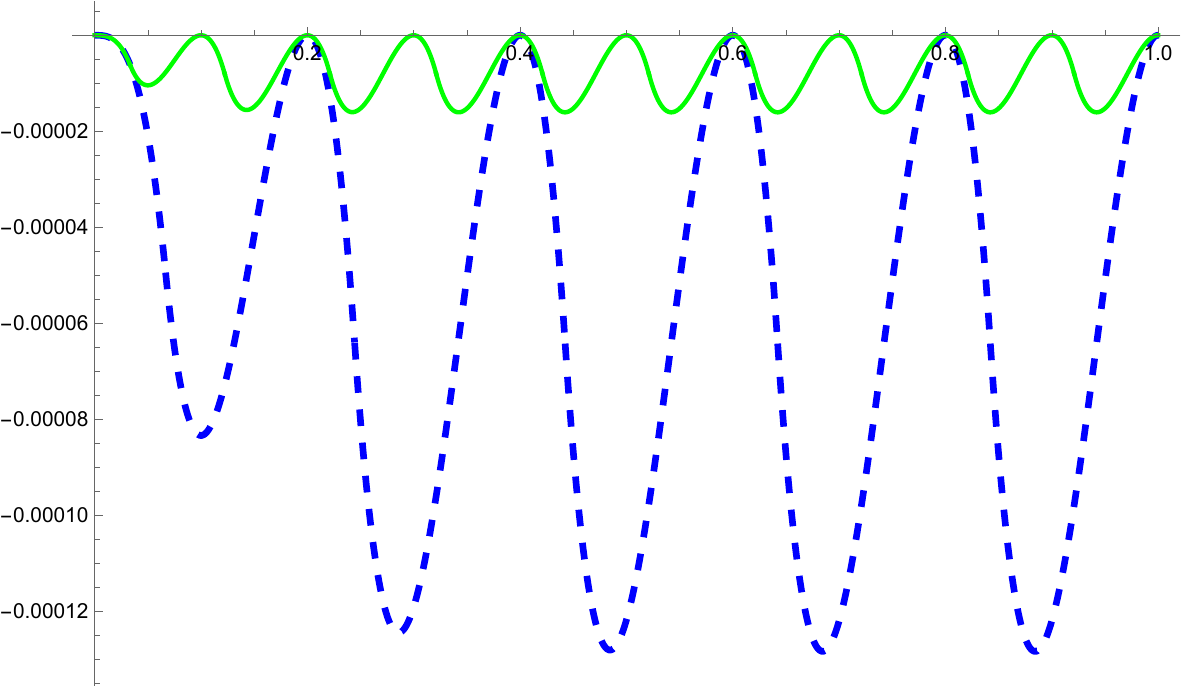} \hspace*{0.2cm}
\includegraphics[scale=0.3,clip]{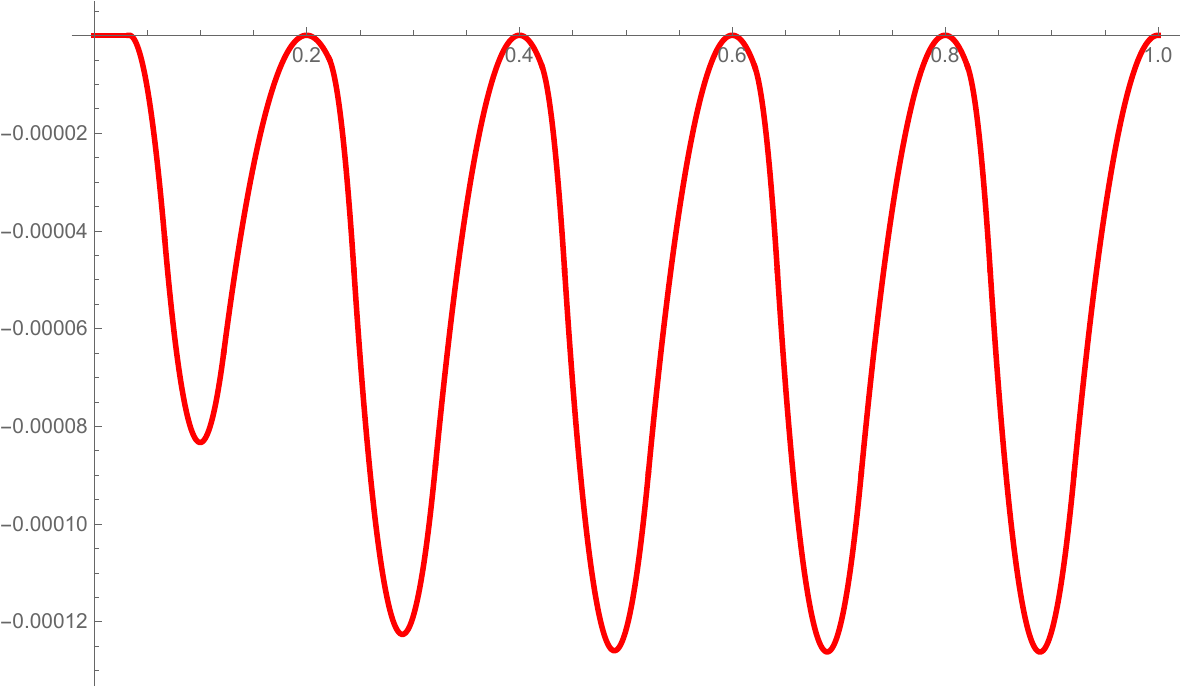}
\label{fig1}
\end{figure}

{\small
\begin{equation*}
\begin{split}
\text{Figure 1. }&\text{ Left: graphs of
$K_3(Q_5^{R,r};t)$ (dashed) and $K_3(Q_{9}^{R,r};t)$ (solid);}\\
&\text{ Right: graph of $K_3(Q_5^{R,r};t)-K_3(Q_{9}^{R,r};t)$.}
\end{split}
\end{equation*}}

\section{Proof of Theorem~\ref{t2}}
\begin{proof}[Proof of Theorem~\ref{t2}(a)] In view of
Remark~\ref{r1}, it suffices to prove the estimates \eqref{e1.14}
only for $R\big[Q_{n+1}^{R,r};f\big]$. We apply the argument from
\cite{KN1995} to the proof of \eqref{e1.6}, comparing
$K_3(Q_{n+1}^{R,r};t)$ with the adjusted one-periodic Bernoulli
monospline
\begin{equation}\label{e4.1}
g(t)=\frac{1}{n^3}\,\Big(B_3(\theta)-B_3\big(\{nt+\theta\}\big)
\Big), \quad \theta=\frac{3+\sqrt{3}}{6}.
\end{equation}
Here, $\{\cdot\}$ is the fractional part function, $B_3$ is the
third Bernoulli polynomial with leading coefficient $1/6$,
$$
B_3(t)=\frac{1}{6}\Big(t^3-\frac{3}{2}t^2+\frac{1}{2}t\Big),
$$
and
$$
B_3(\theta)=-\|B_3\|=-\frac{\sqrt{3}}{216}.
$$
We need the following properties of $g(t)$, defined in \eqref{e4.1}:
\begin{enumerate}[(i)]
\item
The zeros of $g(t)$ in $(0,1)$ are $\ds x_{k}=\frac{k}{n}$,
$k=1,\ldots,n-1$, each of them double;
\item
$g(t)$ satisfies the inequalities
$$
-\frac{\sqrt{3}}{108n^3}\leq g(t)\leq 0,\quad  t\in [0,1]\,;
$$
\item
$g(t)$ has $n$ simple knots in $(0,1)$ located at the points $\ds
\frac{k-\theta}{n},\ k=1,\ldots,n$;
\item
$g(0)=g^{\prime}(0)=g(1)=g^{\prime}(1)=0$ and
$$
g^{\prime\prime}(0+)=g^{\prime\prime}(1-)=\frac{1-2\theta}{2n}<0.
$$
\end{enumerate}
The set of zeros of $K_3(Q_{n+1}^{R,r};t)$ in $(0,1)$ coincides with
that of the zeros of $g(t)$, namely, the double zeros at $x_{k}$,
$k=1,\ldots, n-1$. Furthermore,
\begin{equation}\label{e4.2}
\begin{split}
&K_3(Q_{n+1}^{R,r};0)=K_3^{\prime}(Q_{n+1}^{R,r};0)=
K_3(Q_{n+1}^{R,r};1)=K_3^{\prime}(Q_{n+1}^{R,r};1)=0,\\
&K_3^{\prime\prime}(Q_{n+1}^{R,r};0+)=0,\quad
K_3^{\prime\prime}(Q_{n+1}^{R,r};1-)=a_{n,n}>0.
\end{split}
\end{equation}
Then $s(t)=g(t)-K_3(Q_{n+1}^{R,r};t)$ is a spline function of degree
two with $2n$ knots in $(0,1)$. We apply Lemma~\ref{l2} to $s$ and
obtain
\begin{equation*}
\begin{split}
2n-2&\leq Z_s(0,1)\leq Z_{s^{\prime\prime}}(0,1)+
S^{-}\big(s(0),s^{\prime}(0),s^{\prime\prime}(0+)\big) -
S^{+}\big(s(1),s^{\prime}(1),s^{\prime\prime}(1-)\big)\\
&\leq 2n+S^{-}\Big(0,0,\frac{1-2\theta}{2n}\Big)-
S^{+}\Big(0,0,\frac{1-2\theta}{2n}-a_{n,n}\Big)\\
&\leq 2n-2.
\end{split}
\end{equation*}
Hence, $s(t)$ has no other zeros in $(0,1)$ except the double ones
at $x_{k}$, $k=1,\ldots,n-1$, therefore $s(t)$ does not change its
sign in $(0,1)$. From (iv) and \eqref{e4.2} it follows that
$s(t)\leq 0$ on $[0,1]$, which together with (ii) implies
\begin{equation}\label{e4.3}
-\frac{\sqrt{3}}{108n^3}\leq g(t)\leq K_3(Q_{n+1}^{R,r};t)\leq
0,\quad t\in [0,1]\,.
\end{equation}
If $f\in C^3[0,1]$ and $f^{\prime\prime\prime}(t)\geq 0$, $t\in
[0,1]$, then \eqref{e4.3} implies
\begin{equation*}
\begin{split}
0&\geq R[Q_{n+1}^{R,r};f]=
\int\limits_0^1K_3(Q_{n+1}^{R,r};t)f^{\prime\prime\prime}(t)\,dt
\geq \min_{t\in [0,1]}K_3(Q_{n+1}^{R,r};t)\int\limits_0^1
f^{\prime\prime\prime}(t)\,dt\\
&\geq \min_{t\in [0,1]}g(t)\,\int\limits_0^1
f^{\prime\prime\prime}(t)\,dt = -\frac{\sqrt{3}}{108n^3}
\big(f^{\prime\prime}(1)-f^{\prime\prime}(0)\big).
\end{split}
\end{equation*}
The proof of Theorem~\ref{t2}(a) is complete.
\end{proof}

Figure~2 shows how close to each other are the graphs of the third
Peano kernel of a right Radau quadrature formula and the associated
adjusted Bernoulli monospline $g(t)$ defined in \eqref{e4.1} in the
case $n=4$. For larger $n$, except for a small neighborhood of the
left end-point of the interval, the two graphs are practically
undistinguishable.

\begin{figure}[htp]
\centering
\includegraphics[scale=0.5,clip]{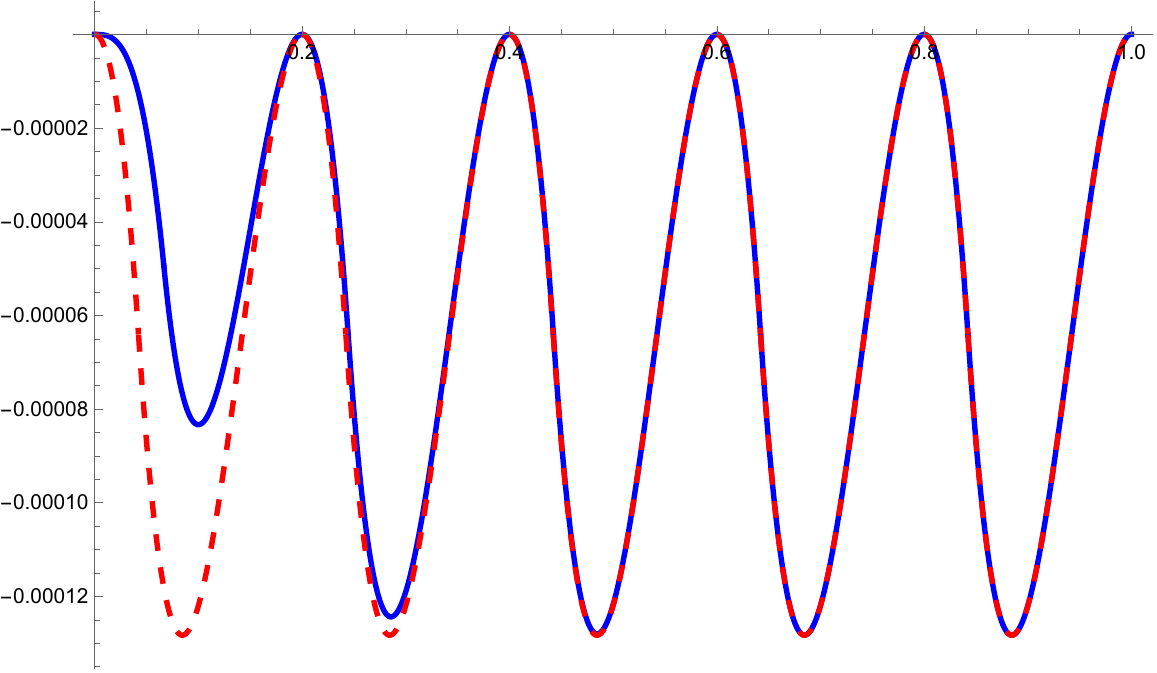}
\label{fig2}
\end{figure}

{\small
\begin{equation*}
\begin{split}
\text{Figure 2.}&\text{ Graphs of Peano kernel
$K_3(Q_{n+1}^{R,r};t)$ (solid) and of the associated} \\
&\text{ adjusted Bernoulli monospline $g(t)$ defined in
(3.4)(dashed), $n=4$.}
\end{split}
\end{equation*}}

\begin{proof}[Proof of Theorem~\ref{t2}(b)]
The argument is similar to that in the proof of part (a). The fourth
Peano kernel of the $(n+1)$-point Gaussian quadrature formula
$$
Q_{n+1}^{G}[f]=\sum_{i=1}^{n+1}a_{i,n+1}^{G}f(\tau_{i,n+1}^{G}),\quad
0<\tau_{1,n+1}<\cdots<\tau_{n+1,n+1}<1,
$$
associated with $S_{n,4}$, is compared with the adjusted Bernoulli
monospline
\begin{equation}\label{e4.4}
g(t)=\frac{1}{n^4}\Big(B_4(\{nt\})-B_4(0)\Big)=
\frac{1}{n^4}\Big(B_4(\{nt\})+\frac{1}{720}\Big),
\end{equation}
where $B_4(t)$ is the fourth Bernoulli polynomial
$$
B_4(t)=\frac{1}{24}\Big(t^4-2t^3+t^2-\frac{1}{30}\Big).
$$
Now $g(t)$ is a monospline of degree four which has $n-1$ simple
knots in $(0,1)$ located at the points $\ds x_{k}=\frac{k}{n}$,
$k=1,\ldots,n-1$. It follows from
$$
-B_4(0)=-B_4(1)=\frac{1}{720}=\|B_4\|
$$
that $g(t)\geq 0$, $t\in (0,1)$, and $g$ has double zeros in $(0,1)$
at $x_{k}$, $k=1,\ldots,n-1$. Moreover,
\begin{equation}\label{e4.5}
\|g\|=\frac{1}{n^4}\Big(\max_{t\in
[0,1]}B_4(t)+\frac{1}{720}\Big)=\frac{1}{384n^4}\,.
\end{equation}

The difference $s(t)=g(t)-K_4(Q_{n+1}^{G};t)$ is a spline function
of degree three with $2n$ simple knots in $(0,1)$, namely, $\ds
\{x_{k}\}_{k=1}^{n-1}\cup \{\tau_{k,n+1}^{G}\}_{k=1}^{n+1}$. We have
\begin{eqnarray*}
&&s(0)=s^{\prime}(0)=s(1)=s^{\prime}(1)=0,\\
&&s^{\prime\prime}(0)=s^{\prime\prime}(1)=\frac{1}{12n^2},\\
&&s^{\prime\prime\prime}(0+)=-\frac{1}{2n},\quad
s^{\prime\prime\prime}(1-)=\frac{1}{2n}.
\end{eqnarray*}
The explicit form of $s^{\prime\prime\prime}(t)$ for $t\in (0,1)$ is
$$
s^{\prime\prime\prime}(t)=-\frac{1}{2n}-
\frac{1}{n}\sum_{k=0}^{n-1}\big(t-x_{k}\big)_{+}^{0}+
\sum_{k=1}^{n+1}a_{k,n+1}^{G}\big(t-\tau_{k,n+1}^{G}\big)_{+}^{0}.
$$
Taking into account that all Gaussian weights $a_{k,n+1}^{G}$ are
positive, we conclude that
$$
Z_{s^{\prime\prime\prime}}(0,1)\leq 2n-1\,.
$$
By applying Lemma~\ref{l2} we obtain
\begin{equation*}
\begin{split}
2n-2\leq Z_s(0,1)\leq& Z_{s^{\prime\prime\prime}}(0,1)+
S^{-}\big(s(0),s^{\prime}(0),s^{\prime\prime}(0),s^{\prime\prime\prime}(0+)\big)\\
&-
S^{+}\big(s(1),s^{\prime}(1),s^{\prime\prime}(1),s^{\prime\prime\prime}(1-)\big)\\
&\leq 2n-1+S^{-}\Big(0,0,\frac{1}{12n^2},-\frac{1}{2n}\Big)-
S^{+}\Big(0,0,\frac{1}{12n^2}, \frac{1}{2n}\Big)\\
&= 2n-2.
\end{split}
\end{equation*}
Hence, the only zeros of $s(t)$ in $(0,1)$ are the double zeros at
$x_{k}$, $k=1,\ldots,n-1$, and $s(t)$ does not change its sign in
$(0,1)$. Since $s(0)=s^{\prime}(0)=0$ and $s^{\prime\prime}(0)>0$,
it follows that $s(t)\geq 0$ on $[0,1]$, hence
\begin{equation}\label{e4.6}
g(t)\geq K_4(Q_{n+1}^{G};t)\geq 0,\quad t\in [0,1].
\end{equation}
If $f\in C^4[0,1]$ and $f^{(4)}(t)\geq 0$ on $[0,1]$, then
\eqref{e4.4} and \eqref{e4.6} imply
\begin{equation*}
\begin{split}
0&\leq
R\big[Q_{n+1}^{G};f\big]=\int\limits_{0}^{1}K_4(Q_{n+1}^{G};t)f^{(4)}(t)\,dt
\leq \max_{t\in [0,1]}g(t)
\int\limits_{0}^{1}f^{(4)}(t)\,dt\\
&=\frac{1}{384n^4}\Big(f^{\prime\prime\prime}(1)-f^{\prime\prime\prime}(0)\Big).
\end{split}
\end{equation*}
The proof of Theorem~\ref{t2}(b) is complete.
\end{proof}
\begin{remark}\label{r4}
Using Lemma~\ref{l2}, error estimates analogous to those in
Theorem~\ref{t2} can be proved for all Gauss-type quadrature
formulae associated with the spaces $S_{n,r}$, $r>4$, defined in
\eqref{e1.8}. However, since the Gauss-type quadrature formulae are
not known for $r>4$, these estimates are not of practical
importance.
\end{remark}

\section*{Acknowledgement} This research is supported by the
Bulgarian National Research Fund under Contract KP-06-N62/4.


\begin{thebibliography}{99}
\bibitem{BB1}
B.\,D. Bojanov, Uniqueness of the monosplines of least deviation,
in: \textit{Numerische Integration} (G. H\"ammerlin, Ed.), ISNM 45,
Birkh\"auser, Basel, 1979, 67--97.

\bibitem{BB2}
B.\,D. Bojanov, Existence and characterization of monosplines of
least $L_p$ deviation, in: \textit{Constructive Function Theory '77}
(Bl. Sendov and D. Va\v{c}ov, Eds.), Sofia, BAN, 1980, 249--268.

\bibitem{BB3} Bojanov, B.\,D.: Uniqueness of the optimal nodes of
quadrature formulae, \textit{Math. Comput.} 36, 1981, 525--546.

\bibitem{LB1955}
L. Brand, A sequence defined by a difference equation, \textit{Amer.
Math. Monthly} 62, 1955, 489--492.

\bibitem{HB1977}
H. Brass, \textit{Quadraturverfahren}, Vandenhoeck \& Ruprecht,
G\"{o}ttingen, 1977.

\bibitem{KJ1976}
K. Jetter, Optimale Quadraturformeln mit semidefiniten Kernen,
\textit{Numer. Math.} 25, 1976, 239--249.

\bibitem{KM1972}
S. Karlin, C. Micchelli, The fundamental theorem of algebra for
monosplines satisfying boundary conditions, \textit{Isr. J. Math.
Soc.} 11, 1972, 405--451.

\bibitem{KN1995a}
P. K\"ohler, G. Nikolov, Error bounds for Gauss type quadrature
formulae related to spaces of splines with equidistant knots,
\textit{J. Approx. Theory} 81, 1995, 368--388.

\bibitem{KN1995}
P. K\"ohler, G. Nikolov, Error bounds for optimal definite
quadrature formulae, \textit{J. Approx. Theory} 81, 1995, 397--405.

\bibitem{GL1977}
G. Lange, \textit{Beste und optimale definite Quadraturformeln}, PhD
Thesis, Technical University Clausthal, 1977.

\bibitem{GL1979}
G. Lange, Optimale definite Quadraturformeln, in \textit{Numerische
Integration} (G. H\"{a}mmerlin, Ed.), Birkh\"{a}user Verlag, Basel,
1979, pp. 187--197.

\bibitem{GN1993}
G. Nikolov, Gaussian quadrature formulae for splines, in:
\textit{Numerische Integration, IV} (G. H\"{a}mmerlin and H. Brass
Eds.), ISNM 112, Birkh\"{a}user Verlag, Basel, 1993, 267--281.

\bibitem{GN1996}
G. Nikolov, On certain definite quadrature formulae, \textit{J.
Comp. Appl. Math.} 75, 1996, 329--343.


\bibitem{GP:1913}
G. Peano, Resto nelle formule di quadratura espresso con un
integrale definito. \emph{Atti della Reale Accademia dei Lincei:
Rendiconti (Ser. 5)} 22, 1913, 562--569.

\bibitem{GS1972}
G. Schmeisser, Optimale Quadraturformeln mit semidefiniten Kernen,
\textit{Numer. Math.} 20, 1972, 32--53.

\bibitem{AZ1}
A. Zhensykbaev, Best quadrature formulae for some classes of
periodic differentiable functions, \textit{Izv. Akad. Nauk SSSR Ser.
Mat.} 41, 1977 (in Russian); English Translation in: \textit{Math.
USSR Izv.} 11, 1977, 1055--1071.

\bibitem{AZ2}
A. Zhensykbaev, Monosplines and optimal quadrature formulae for
certain classes of non-periodic functions, \textit{Anal. Math.} 5,
1979, 301--331 (in Russian).
\end{thebibliography}
\end{document}